\documentclass[reqno, centertages,12pt]{amsart}

\usepackage{amsfonts, amssymb, latexsym, amscd, epsf, cite, gensymb}
\usepackage{amsmath, amsthm}
\usepackage[colorlinks=true, pdfstartview=FitV, linkcolor=blue, citecolor=blue, urlcolor=blue,pagebackref=false]{hyperref}
\usepackage[mathscr]{eucal}
\usepackage{MnSymbol}
\usepackage{color}
\usepackage{graphicx}

\usepackage{enumerate}

\usepackage{esint}
\usepackage{mathtools}
\usepackage{microtype}

\numberwithin{equation}{section}
\usepackage{verbatim}

\theoremstyle{plain}
\newtheorem{thm}{Theorem}[section]

\newtheorem{lem}[thm]{Lemma}
\newtheorem{prop}[thm]{Proposition}

\newtheoremstyle{named}{}{}{\itshape}{}{\bfseries}{.}{.5em}{\thmnote{#3's }#1}
\theoremstyle{named}

\theoremstyle{definition}
\newtheorem{defn}[thm]{Definition}

\definecolor{darkgreen}{rgb}{0,0.5,0}
\newcommand{\jlcomment}[1]{\marginpar{\raggedright\scriptsize{\textcolor{darkgreen}{#1}}}}

\newcommand{\RR}{\mathbb{R}}

\newcommand{\eps}{\varepsilon}

\newcommand{\al}{\alpha}

\newcommand{\ep}{\epsilon}

\newcommand{\la}{\lambda}

\newcommand{\ch}{\text{\rm ch}}

\newcommand{\supp}{\text{\rm supp}}
\newcommand{\interior}{\text{\rm int}}

\newcommand{\rst}[1]{\ensuremath{{\mathbin\upharpoonright}%
\raise-.5ex\hbox{$#1$}}}

\DeclareMathOperator{\card}{card}
\newcounter{gscan}
\newcounter{btscan}
\newcounter{cscan}
\newcounter{hscan}
\newcounter{fscan}
\newcounter{pscan}
\newcounter{sscan}

\DeclareMathOperator*{\esssup}{\text{\rm ess\,supp}}

\usepackage[foot]{amsaddr}
 
\title[Multidimensional Fisher-KPP transition fronts]{Multidimensional transition fronts for Fisher-KPP reactions}

\author[Alwan et al]{Amir Alwan}
\author[]{Zonglin Han}
\author[]{Jessica Lin}
\address[AA, ZH, JL]{Department of Mathematics, University of Wisconsin--Madison, 480 Lincoln Dr., Madison, WI 53706}
\email[Alwan, Han, Lin]{alwan@wisc.edu, zhan29@wisc.edu, jessica@math.wisc.edu}

\author[]{Zijian Tao}
\address[ZT]{Department of Mathematics, California Institute of Technology, 1200 E. California Blvd., Pasadena, CA 91125}
\email[Tao]{ztao@caltech.edu}

\author[]{Andrej Zlato\v{s}}
\address[AZ]{Department of Mathematics, University of California San Diego, 9500 Gilman Dr. \# 0112, La Jolla, CA 92093}
\email[Zlato\v{s}]{zlatos@ucsd.edu}

\begin{document}

\begin{abstract}
We study entire solutions to homogeneous reaction-diffusion equations in several dimensions with Fisher-KPP reactions.  
Any entire solution $0<u<1$ is known to satisfy
\[
\lim_{t\to -\infty} \sup_{|x|\le c|t|} u(t,x) = 0 \qquad \text{for each $c<2\sqrt{f'(0)}\,$,}
\]
and we consider here those satisfying 
\[
\lim_{t\to -\infty} \sup_{|x|\le c|t|} u(t,x) = 0 \qquad \text{for some $c>2\sqrt{f'(0)}\,$.}
\]
When $f$ is $C^2$ and concave, our main result provides an almost complete characterization of transition fronts as well as transition solutions with bounded width within this class of solutions. 

\end{abstract}

\maketitle

\section{Introduction}
In this paper we study entire solutions of reaction-diffusion equations
\begin{equation}\label{RxnDiffEq}
u_t=\Delta u+f(u)\qquad \text{on }\mathbb{R}\times \mathbb{R}^{d},
\end{equation}
with Fisher-KPP reaction functions $f\in C^{1+\gamma}([0,1])$ for some  $\gamma>0$.  Specifically, we also assume that
\begin{equation} \label{hyp:flip}
 f(0)=f(1)=0, \qquad  f^{\prime}(0)=1, \qquad 
 \text{$0<f(u)\le u$ on $(0,1)$.}
\end{equation}

We note that a simple scaling argument extends our results to the general Fisher-KPP case
\[
 f(0)=f(1)=0, \qquad  f^{\prime}(0)>0, \qquad 
 \text{$0<f(u)\le f'(0)u$ on $(0,1)$.}
\]

The study of \eqref{RxnDiffEq} was started 80 years ago by Kolmogorov, Petrovskii, and Piskunov \cite{kpp} and Fisher \cite{fisheradv} in one dimension $d=1$, while here we consider entire solutions $u:\mathbb R^{d+1}\to[0,1]$ for any $d\ge 1$.  These model propagation of reactive processes such as forest fires, nuclear reactions in stars, or population dynamics.  The value $u=0$ represents the unburned (or minimal-temperature or zero-population-density) state, while $u=1$ represents the burned (or maximal-temperature or maximal-population-density) state.  
Fisher-KPP reactions possess the ``hair-trigger effect'', meaning that for any non-zero solution $0\le u \le 1$, the asymptotically stable state $u=1$ will invade the whole spatial domain $\mathbb R^d$ as $t\to\infty$ (while the state $u=0$ is unstable).  In fact, we have \cite{AW}
\begin{equation} \label{9.0}
\lim_{t\to\infty} \inf_{|x|\le ct} u(t,x) = 1 \qquad \text{for each $c<2$.}
\end{equation}
This immediately implies that except when $u\equiv 1$, we also have
\begin{equation} \label{9.1}
\lim_{t\to -\infty} \sup_{|x|\le c|t|} u(t,x) = 0 \qquad \text{for each $c<2$.}
\end{equation}
Note that the strong maximum principle and $0\le u\le 1$ imply that $0<u<1$ whenever $u\not\equiv 0,1$, and we will assume this from now on.

In their pioneering work \cite{hamnad.mani}, Hamel and Nadirashvili provided a partial characterization of such solutions of \eqref{RxnDiffEq}.  Under the additional hypotheses of $f\in C^2([0,1])$, $f$ being concave, and $f'(1)<0$, they identified all solutions $u:\mathbb R^{d+1}\to(0,1)$ which also satisfy (cf. \eqref{9.1})
\begin{equation} \label{9.2}
\lim_{t\to -\infty} \sup_{|x|\le c|t|} u(t,x) = 0 \qquad \text{for some $c>2$}
\end{equation}
(we will call these {\it Hamel-Nadirashvili solutions}).
They showed that these solutions are naturally parametrized by all finite positive Borel measures supported inside the open unit ball in $\mathbb R^d$.  One of us later showed \cite{Zlatos.kpp} that this infinite-dimensional manifold of solutions, parametrized by Borel measures, also exists without the additional hypotheses from \cite{hamnad.mani} (see Theorem~\ref{ZlatosThm:tf} below), although it is not yet known whether other solutions satisfying \eqref{9.2} can exist in this case.

It follows from \eqref{9.0} and \eqref{9.1} that all entire solutions $0<u<1$ for Fisher-KPP reactions satisfy
\begin{equation} \label{9.3}
\lim_{t\to -\infty} u(t,x) =0 \qquad\text{and} \qquad \lim_{t\to \infty} u(t,x) =1
\end{equation}
locally uniformly.    Our goal here is to study the nature of this transition from 0 to 1.  Aerial footage of forest fires usually shows relatively narrow lines of fire separating burned and unburned areas, and we investigate the question which entire solutions also have this property.  More specifically, which are {\it transition fronts}, defined by Berestycki and Hamel in \cite{berhamel1, berhamel2} (and earlier in some special situations by Matano  \cite{Matano} and Shen  \cite{Shen}); and more generally, which are {\it transition solutions with bounded width}, defined by one of us in \cite{Zlatos.bw}.  Let us now state these definitions.

For any $u$ as above, $t\in\mathbb R$, and $\ep\in [0,1]$ let
\begin{align*}
\Omega_{u,\epsilon}(t) &:=\{x\in\mathbb{R}^d: u(t,x)\ge \epsilon\}, 
\\ \Omega_{u,\epsilon}'(t) & :=\{x\in\mathbb{R}^d: u(t,x)\le \epsilon\}, 
\end{align*}
and for any $E\subseteq \RR^{d}$ and $L>0$ let
\begin{equation*}
B_{L}(E):=\bigcup_{x\in E} B_{L}(x).
\end{equation*}

\begin{defn}\label{DefOfTransitionFront}
Let $0<u<1$ be an entire solution to \eqref{RxnDiffEq}.
	\begin{enumerate}[(i)]
		\item $u$ is  a \textit{transition solution} if it satisfies \eqref{9.3} locally uniformly.
\item  $u$ has \textit{bounded width} if   for each $\epsilon\in\left(0,\frac{1}{2}\right)$ there is $L_{\epsilon}<\infty$ such that 
		\begin{align} \label{1.1}
		\Omega_{u,\epsilon}(t)\subseteq B_{L_{\epsilon}}(\Omega_{u,1-\epsilon}(t)) \qquad \text{for each $t\in\mathbb{R}$.}
		\end{align}
\item  $u$ is a \textit{transition front} if it has bounded width,  for each $\epsilon\in\left(0,\frac{1}{2}\right)$ there is $L_{\epsilon}'<\infty$ such that 
		\begin{align} \label{1.2}
		\Omega_{u,1-\epsilon}'(t)\subseteq B_{L_{\epsilon}'}(\Omega_{u,\epsilon}'(t)) \qquad \text{for each $t\in\mathbb{R}$,}
		\end{align}
and there are $n,L$ such that for any $t\in\mathbb{R}$, there is a union $\Gamma_t$ of at most $n$ rotated continuous graphs in $\RR^d$ which satisfy 
\[
\partial\Omega_{u,1/2}(t)\subseteq B_{L}(\Gamma_t).
\]
	\end{enumerate}

\end{defn}

\textit{Remarks.} 1.  Recall that when $f$ is Fisher-KPP, then all entire solutions $0<u<1$ are transition solutions.
\smallskip

2. A {\it rotated continuous graph} in $\RR^d$ is a rotation of the graph of some continuous function $h:\RR^{d-1}\to \RR$ (which is a subset of $\RR^d$).
\smallskip

3.  The original definition of transition fronts in  \cite{berhamel1, berhamel2} was slightly different from (iii), but the two are equivalent \cite{Zlatos.bw}.
\smallskip

4. In one dimension $d=1$ the set $\Gamma_t$ in (iii) is just a collection of at most $n$ points.  The special case $n=1$ of transition fronts with a {\it single interface} is of particular interest and has recently been studied extensively for various types of reactions  (see, e.g., \cite{berhamel2, DHZ2, hamnad, HamRos, HamRos2, LimZla, MNRR, MRS, Nadin, NRRZ, NolRyz, Shen, Shen1, ShenShen, ShenShen2, TZZ, VakVol, Zlatos.kpp, ZlaGenfronts, ZlaBist}).  These are entire solutions $0<u<1$ satisfying
\begin{equation} \label{9.5}
\lim_{x\to-\infty} u(t, x+x_t)=1 \qquad\text{and}\qquad \lim_{x\to\infty} u(t, x+x_t)=0
\end{equation}
uniformly in $t\in\RR$, where $x_t:=\max \{x\in\RR: u(t,x)=\tfrac 12\}$ (or with 0 and 1 exchanged in \eqref{9.5}).  They were introduced as a generalization of the concept of {\it traveling fronts}, solutions of the form $u(t,x)=U(x-ct)$ for some decreasing front profile $U:\RR\to (0,1)$ with $\lim_{s\to-\infty}U(s)=1$  and $\lim_{s\to\infty}U(s)=0$, and some front speed $c$. (It is well-known that for \eqref{RxnDiffEq} with a Fisher-KPP reaction, these exist if and only if $c\ge 2\sqrt{f'(0)}$.)  Traveling fronts, which were already studied in \cite{kpp,fisheradv}, only exist for homogeneous reactions, and transition fronts are their natural generalization that can exist in both homogeneous and heterogeneous (i.e., $x$-dependent) media.  We discuss recent results concerning transition fronts for homogeneous reactions below.
\smallskip

5.  Solutions satisfying \eqref{1.1} and \eqref{1.2} but not necessarily the closeness-to-graphs condition are said to have {\it doubly bounded width} \cite{Zlatos.bw}.  Our main result (Theorem \ref{ourmainthm}) and its proof remain unchanged when ``transition fronts'' are replaced by ``transition solutions with doubly bounded width''.
\smallskip

It is easily seen that a transition solution $u$ is a transition front if and only if the Hausdorff distance of any two {\it level sets} $\{x\in\RR^d:u(t,x)=\epsilon\}$ of $u$ stays bounded uniformly in time, and the level set $\{x\in\RR^d:u(t,x)=\frac 12\}$ (then also any other) is at each time uniformly close to a uniformly bounded number of time-dependent rotated continuous graphs.  In contrast, $u$ is a transition solution with bounded width if and only if the Hausdorff distance of any two {\it super-level sets} $\Omega_{u,\epsilon}(t)$ of $u$ stays bounded uniformly in time.  

This distinction results in some notable differences.  For instance, transition fronts (and transition solutions with doubly bounded width) satisfy 
\begin{equation} \label{9.4}
\inf_{x\in \RR^{d}} u(t,x)=0\quad\text{and}\quad \sup_{x\in \RR^{d}} u(t,x)=1
\end{equation}
for each $t\in \RR$, while transition solutions with bounded width need not.  Also, transition solutions with bounded width in dimensions $d\ge 2$ may involve dynamics where the invading state $u\approx 1$ first encircles  large regions  where $u\approx 0$ (with their sizes unbounded as $t\to\infty$) and then invades them.  On the other hand,  such solutions cannot be transition fronts (or have doubly bounded width) because, for instance, at some time $t$ there will be a point $x$ with $u(t,x)=\frac 23$ near the center of such a region but points $y$ with $u(t,y)=\frac 13$ will all lie outside of this region (and thus far away from $x$).  Because this phenomenon does occur for various heterogeneous reactions (e.g., for stationary ergodic reactions with short-range correlations), preventing existence of transition fronts in these settings, it is important to study both these classes of solutions to \eqref{RxnDiffEq}.  We refer to \cite{Zlatos.bw} for a more detailed discussion of the relevant issues.

Coming back to the homogeneous equation \eqref{RxnDiffEq} with a Fisher-KPP reaction $f$, the first systematic study of its entire solutions was undertaken in \cite{hamnad,hamnad.mani} under some additional conditions on $f$.  We will use here the following closely related result from  \cite{Zlatos.kpp}, which concerns the main object of our study --- the Hamel-Nadirashvili solutions to \eqref{RxnDiffEq} --- and holds for  general Fisher-KPP reactions.  In order to state it, first recall that if $\mu$ is a positive Borel measure on $\RR^d$, its support $\supp (\mu)$ is the minimal closed set $A$ such that $\mu(A^c)=0$, while its essential support is any Borel set $A$ such that $\mu(A)=\mu(\mathbb{R}^d)$ and $\mu(A^\prime)<\mu(A)$ whenever $A^{\prime}\subseteq A$ and $A\setminus A^{\prime}$ has positive Lebesgue measure. The collection of all essential supports of $\mu$ will be denoted $\esssup(\mu)$.  Following \cite{Zlatos.kpp}, we then define the \textit{convex hull} of  $\mu$  to be
	\begin{align*}
	\ch(\mu):=\bigcap_{A\in\esssup(\mu)}\ch(A),
	\end{align*}
where $\ch(A)$ is the convex hull of the set $A$.  Note that we may have $\ch(\mu)\notin\esssup(\mu)$ \cite{Zlatos.kpp}. Finally, let $B_r$ denote the open ball $B_r(0)\subseteq \RR^d$ with radius $r$ and centered at 0, and let $S^{d-1}:=\partial B_1$.

\begin{thm}[\hskip0.01mm\cite{Zlatos.kpp}]\label{ZlatosThm:tf}
Assume that $f\in C^{1+\gamma}([0,1])$ for some  $\gamma>0$ and satisfies \eqref{hyp:flip}, let $\mu$ be a finite positive non-zero Borel measure on $\mathbb{R}^d$ with $\supp(\mu)\subseteq B_1$, and let 
	\begin{equation}\label{EquationOfV}
	v_{\mu}(t,x):=\int_{B_1} e^{-\xi\cdot x +(|\xi|^2+1)t}\,d\mu(\xi).
	\end{equation}
	\begin{enumerate}[(i)]
		\item There is an increasing function $h:[0,\infty]\rightarrow[0,1)$ with $h(0)=0$, $h^{\prime}(0)=1$ and $\lim\limits_{v\rightarrow\infty}h(v)=1$, and an entire solution $u_{\mu}$ of \eqref{RxnDiffEq} such that $(u_{\mu})_t>0$ and 
		\begin{align}
		h(v_{\mu})\le u_{\mu}\le\min\{v_{\mu},1\}.\label{ComparisonOfVandU}
		\end{align}
		In addition, $u_{\mu}\not\equiv u_{\mu^{\prime}}$ whenever $\mu\ne \mu^{\prime}$.
		\item We have
		\begin{align}\label{Znec}
		\inf_{x\in\mathbb{R}^d} u_{\mu}(t,x)=0\qquad\text{and}\qquad \sup_{x\in\mathbb{R}^d} u_{\mu}(t,x)=1
		\end{align}
		for each $t\in\mathbb{R}$ if and only if $0\not\in \ch(\mu)$.
		\item If $0\not\in \supp(\mu)$, then $u_\mu$ has bounded width.
	\end{enumerate}
\end{thm}

{\it Remarks.}  1.  If also $f\in C^2([0,1])$, it is concave, and $f'(1)<0$, then \cite[Theorem 1.2]{hamnad.mani} shows that the solutions from (i) are all those entire solutions $0<u<1$ satisfying \eqref{9.2}.  We note that in this case \cite{hamnad.mani} also constructs entire solutions corresponding to some measures supported in $\bar B_1$ but not in $B_1$ (which then do not satisfy \eqref{9.2}), namely those whose restriction to $S^{d-1}$ is a finite sum of Dirac masses.
\footnote{In fact, the measures in \cite{hamnad.mani} are supported in $B_2^c\cup\{\infty\}$ but the map $\xi\mapsto(1+|\xi|^{-2})\xi$ establishes the relevant correspondence between $\bar B_1$ and $B_2^c\cup\{\infty\}$.}
\smallskip

2.  Note that the functions $e^{-\xi\cdot x +(|\xi|^2+1)t}$ and $v_\mu$ from \eqref{EquationOfV} solve the linearization
\[
v_t=\Delta v + v
\]
of \eqref{RxnDiffEq} at $u=0$.  Moreover, if we denote $c_{|\xi|}:=|\xi|+\frac{1}{|\xi|}$ for $\xi\neq 0$, then 
\[
e^{-\xi\cdot x +(|\xi|^2+1)t} = e^{-\xi\cdot x +|\xi|c_{|\xi|}t} = e^{-\xi\cdot (x -\frac \xi{|\xi|}c_{|\xi|}t)}.
\]
So this is an exponential that moves with speed $c_{|\xi|}$ in the direction $\frac\xi{|\xi|}$.
\smallskip

3.  (ii) and \eqref{9.4} show that $0\not\in\ch(\mu)$ is a necessary condition for $u_{\mu}$ to be a transition front.
\smallskip

4. This result, and thus also Theorem \ref{ourmainthm} below, holds for $f$ satisfying \eqref{hyp:flip} which is only Lipschitz, as long as $f(u)\ge  g(u)$ on $[0,1]$ for some $g\in C^{1}([0,1])$  such that  $g(0)=g(1)=0$, $g^{\prime}(0)=1$, $g(u)>0$  and $g^{\prime}(u)\le 1$ on $(0,1)$, and $\int_{0}^{1}\frac{u-g(u)}{u^2}du<\infty$ \cite{Zlatos.kpp}.  We note that if $f\in C^{1+\gamma}([0,1])$ satisfies \eqref{hyp:flip}, then there exists such function $g$ with $g(u)=u-Cu^{1+\gamma}$ for some $C$ and all small $u\ge 0$.
\smallskip

We now turn to our main result, an almost complete characterization of  
 transition fronts as well as  transition solutions with bounded width within the class of the solutions from Theorem \ref{ZlatosThm:tf}.  Recall that if $f\in C^2([0,1])$ is concave and \hbox{$f'(1)<0$}, then this class coincides with the class of Hamel-Nadirashvili solutions.  In one dimension $d=1$ and under these extra hypotheses, a complete characterization of transition fronts among all the solutions from \cite{hamnad.mani} (these are then parametrized by finite positive non-zero Borel measures $\mu$ on the interval $[-1,1]=\bar B_1$) was recently obtained by Hamel and Rossi \cite{HamRos2}. 
They proved that the solution $u_\mu$ is a transition front if and only if $\supp(\mu)\subseteq [-1,0)$ or $\supp(\mu)\subseteq (0,1]$.  In several dimensions, this task is considerably more challenging because the geometry of $B_1$ is more complicated there.  In fact, we are not aware of any relevant previous results for Fisher-KPP reactions.  We note that transition fronts and transition solutions with bounded width for ignition and bistable reactions satisfying very mild hypotheses were proved to increase in time  \cite{berhamel2, Zlatos.bw}, and examples of transition fronts  for homogeneous bistable reactions that are not traveling fronts were recently constructed in \cite{Hamel}.

For $\zeta\in S^{d-1}$ and $\alpha\in[0,1]$, let
\[
\mathcal{W}_{\alpha, \zeta}:=\{x\in\mathbb{R}^d: x\cdot \zeta\ge \alpha|x|\},
\]
which is a closed cone with axis $\zeta$ when $\alpha>0$, while $\mathcal{W}_{0, \zeta}$ is the closed half-space with inner normal $\zeta$.

We will also call an \textit{upright cone}
\begin{align}\label{uprightcone}
\mathcal{W}_{\alpha}:=\mathcal{W}_{\alpha,e_d}=\{x\in\mathbb{R}^d: x_d\ge \alpha |x|\}.
\end{align}

\begin{thm}\label{ourmainthm}
Let $f,\mu,u_\mu$ be as in Theorem \ref{ZlatosThm:tf}. 
\begin{enumerate}[(i)]
\item If  there are $\zeta\in S^{d-1}$ and $\alpha>0$ such that 
\begin{equation}\label{h:1'}\tag{H1}
0\notin\supp(\mu)\subseteq \mathcal{W}_{\al, \zeta}, 
\end{equation}
then $u_{\mu}$ is both a transition front and a transition solution with bounded width. 
\item 
If there are $\zeta\in S^{d-1}$ and $\alpha>0$ such that 
\begin{equation*}\label{h:2}\tag{H2}
0\in \supp(\mu)\subseteq \mathcal{W}_{\alpha,\zeta}, 
\end{equation*}
 then $u_{\mu}$ is neither a transition front nor a transition solution with bounded width. 
\item If 
\begin{equation*}\label{h:3}\tag{H3}
\supp(\mu)\not\subseteq \mathcal{W}_{0, \zeta}\quad\text{for each } \zeta\in S^{d-1},
\end{equation*}
then $u_{\mu}$ is a transition solution with bounded width but not a transition front. 
\end{enumerate}
\end{thm}

Notice that the only cases of measures from Theorem \ref{ZlatosThm:tf} not covered by this result are those supported in some half-space $\mathcal{W}_{0, \zeta}$ but not in any cone $\mathcal{W}_{\alpha,\zeta}$ with $\alpha>0$.  We can still say something in this case:  if $0\notin\supp(\mu)$, then Theorem \ref{ZlatosThm:tf}(iii) shows that $u_\mu$ is a transition solution with bounded width, and we also conjecture that $u_\mu$ is not a transition front.  However, if $0\in\supp(\mu)$, then determining whether $u_\mu$ is a transition front and/or a transition solution with bounded width will likely be a very delicate question.

We prove the three parts of  Theorem \ref {ourmainthm} in the following three sections,leaving some technical lemmas for the Appendix.

\smallskip

{\bf Acknowledgements:} AZ was supported in part by NSF grants DMS-1147523, DMS-1656269, and DMS-1652284.  AA and JL were supported in part by NSF grant DMS-1147523.  ZH and ZT were supported in part by NSF grant DMS-1656269.  AA, ZH, and ZT gratefully acknowledge the hospitality of the Department of Mathematics at the University of Wisconsin--Madison during the REU ``Differential Equations and Applied Mathematics'', where this research  originated.

\section{Proof of Theorem \ref{ourmainthm}(i)}\label{s:(i)}

We may assume without loss of generality that  $\zeta=e_d$, so that the cone $\mathcal{W}_{\al, \zeta}=\mathcal{W}_{\al}$ is upright.  Then \eqref{h:1'} implies there is $\delta>0$ such that 
\begin{equation}\label{h:1} 
\supp(\mu)\subseteq \mathcal{W}_{\al} \cap A(\delta, 1),
\end{equation}
with $A(r_1,r_2):=B_{r_2}\setminus B_{r_1}$ an annulus.
In particular,  
\begin{equation*}
 \inf\{x_d:x\in\supp(\mu)\}\ge \alpha\delta>0.
\end{equation*}	
Let us first show that $u_{\mu}$ has bounded width (recall that each $u_\mu$ is a transition solution). This follows immediately from Theorem \ref{ZlatosThm:tf}
but our argument will also be  useful in the proof that $u_\mu$ is a transition front.
Let $\epsilon\in\left(0,\frac{1}{2}\right)$ and $x\in \Omega_{u_{\mu},\epsilon}(t)$, and define $s:=(\alpha\delta)^{-1}\ln(h^{-1}(1-\epsilon)/\epsilon)\geq 0$ and $x_s:=x-se_d$.  Here $h$ is the $\mu$-dependent function from Theorem \ref{ZlatosThm:tf}(i).
From \eqref{h:1}  we have
\[
	v_{\mu}(t,x_s)
	=\int_{B_1} e^{s(\xi\cdot e_d)} e^{-\xi\cdot x+(|\xi|^2+1)t}\,d\mu(\xi)
	\ge e^{s \alpha\delta}\int_{B_1}e^{-\xi\cdot x+(|\xi|^2+1)t}\,d\mu(\xi),
\]
so the definition of $s$ and  \eqref{ComparisonOfVandU} yield
\[	
	v_{\mu}(t,x_s)\ge \frac{h^{-1}(1-\epsilon)}\epsilon v_{\mu}(t,x) \ge  \frac{h^{-1}(1-\epsilon)}\epsilon u_{\mu}(t,x)\ge  h^{-1}(1-\epsilon).
\]
From \eqref{ComparisonOfVandU} we now have
$x_s\in \Omega_{u_{\mu},1-\epsilon}(t)$, so \eqref{1.1} with $u=u_\mu$ holds
for each $t\in\RR$ and $L_\epsilon:=s+1$.  
Hence $u_\mu$ is a transition solution with  bounded width.
	
The verification of \eqref{1.2} 
for $u_\mu$ is analogous.
If $\ep\in\left(0,\frac{1}{2}\right)$ and $x\in \Omega_{u_{\mu},1-\epsilon}'(t)$, then the above argument for $x_s:=x+se_d$ yields
\[	
	v_{\mu}(t,x_s)\le \frac \epsilon{h^{-1}(1-\epsilon)} v_{\mu}(t,x) \le  \frac \epsilon{h^{-1}(1-\epsilon)} h^{-1}(u_{\mu}(t,x))\le   \epsilon.
\]
From \eqref{ComparisonOfVandU} we now have

$x_s\in\Omega_{u_{\mu},\epsilon}'(t)$, so \eqref{1.2} with $u=u_\mu$ holds for each $t\in\RR$ and $L_\epsilon':=L_\epsilon$.

	Finally,  the last claim in Definition \ref{DefOfTransitionFront}(iii) is satisfied with  $\Gamma_t:=\{x\in\RR^d:v_\mu(t,x)=\frac 12\}$ (which is a graph of a function of $(x_1,\dots,x_{d-1})$ because $\supp(\mu)\subseteq \mathcal{W}_{\al}\cap A(\delta,1)$ implies $(v_\mu)_{x_d}<0$) and any $L> L_{h(1/2)}$. 

	Indeed, if $u_\mu(t,x)=\frac 12$, then \eqref{ComparisonOfVandU} and the above arguments show that $v_\mu(t,x)\ge \frac 12$ as well as 
\[
v_\mu(t,x+L_{h(1/2)}e_d)\le h^{-1}\left(u_\mu(t,x+L_{h(1/2)}e_d)\right)\le h^{-1} \left(h\left(\frac 12\right) \right) =\frac 12.
\]
Hence there is $l\in[0,L_{h(1/2)}]$ such that $x+le_d\in\Gamma_t$,
and it follows that $u_{\mu}$ is indeed a transition front.

\section{Proof of Theorem \ref{ourmainthm}(ii)}\label{s:(ii)}

We again assume without loss that  $\zeta=e_d$, so the cone $\mathcal{W}_{\al, \zeta}=\mathcal{W}_{\al}$ is upright, and let $h$ be the $\mu$-dependent function from Theorem \ref{ZlatosThm:tf}(i).
We will now show that the width of the transition zone of $u_{\mu}$ becomes unbounded as $t\to\infty$, violating Definition \ref{DefOfTransitionFront}(ii). Thus, $u_{\mu}$ is neither a transition solution with bounded width nor a transition front.

First consider the case $\mu(\{0\})>0$ and let $t_{0}:= \ln (2\mu(\{0\}))$.  Then from the Lebesgue dominated convergence theorem, we have 
\[
\lim_{x_d\to \infty} v_\mu(-t_0,x)= \mu(\{0\}) e^{-t_0} \qquad (\le v_{\mu}(-t_0, x) \text{ for all $x\in\mathbb R^d$})
\]
uniformly in $(x_1,\dots,x_{d-1})$.  This and Theorem \ref{ZlatosThm:tf}(i) show that there is $M<\infty$ such that $u_{\mu}(-t_0,x)\in[h(\mu(\{0\}) e^{-t_0}),\frac 12]$ whenever $|x_d|>M$.  Thus $L_{h(\mu(\{0\}) e^{-t_0})}$ from \eqref{1.1} with $u=u_\mu$
cannot be finite and we are done.

Let us now assume $\mu(\{0\})=0$, and fix any $\eps\in(0,\frac 14)$ such that $h^{-1}(\eps)\le \frac 14$.  For each $t\in\mathbb R$, let $X(t)=(0,\dots,0, s_{t})$ be such that $v_\mu(t,X(t))=h^{-1}(\eps)$.  This point is unique because $\supp(\mu)\subseteq \mathcal{W}_{\al}$ and $\mu(\{0\})=0$ imply $(v_\mu)_{x_d}<0$.  

Fix any $\delta\in(0,1)$ and let $\delta'\in(0,\delta)$ be such that $c_{\delta'}\ge \frac 3\alpha c_\delta$.  For instance, $\delta'=\frac{\alpha\delta}6$ works.  Next let
\begin{align*}
v_1(t,x) &:=\int_{A(\delta,1)} e^{-\xi\cdot x +|\xi|c_{|\xi|}t}\,d\mu(\xi), \\
v_2(t,x) &:=\int_{B_\delta} e^{-\xi\cdot x +|\xi|c_{|\xi|}t}\,d\mu(\xi),  \\
v_3(t,x) &:=\int_{B_{\delta'}} e^{-\xi\cdot x +|\xi|c_{|\xi|}t}\,d\mu(\xi),  
\end{align*}
so that $v_\mu=v_1+v_2$.  Note also that  $(v_j)_{x_d}<0$ for $j=1,2,3$.

Let now $r_t:=\frac 2\alpha c_\delta t$ and $Y(t)=(0,\dots,0, r_{t})$.  Then from $c_{|\xi|}$ being decreasing in $|\xi|\in(0,1]$, we obtain for any $\xi\in \mathcal{W}_{\al}\cap A(\delta,1)$ and $t\ge 0$,
\[
-\xi\cdot Y(t)+|\xi|c_{|\xi|} t \le -|\xi|(\alpha r_t -c_\delta t) \le -\delta c_\delta t \le -t.
\]
On the other hand, for $\xi\in \mathcal{W}_{\al}\cap B_{\delta'}$ and $t>0$ we obtain
\[
-\xi\cdot Y(t)+|\xi|c_{|\xi|} t \ge |\xi|( c_{\delta'} t-r_t) \ge \frac{|\xi| c_\delta}\alpha t.
\]
From these, $\mu([\mathcal{W}_{\al}\cap B_{\delta'}]\setminus\{0\})>0$, and the Lebesgue dominated convergence theorem it follows that 
\[
\lim_{t\to\infty} v_1(t,Y(t))=0 \qquad\text{and}\qquad \lim_{t\to\infty} v_3(t,Y(t))=\infty.
\]
Therefore $s_t>r_t$ and $\lim_{t\to\infty} v_1(t,X(t))=0$.  But then from $v_\mu=v_1+v_2$, $|\nabla v_1|\le v_1$, $|\nabla v_2|\le \delta v_2$, and $v_\mu(t,X(t))\le \frac 14$ it follows that 
\[
\lim_{t\to\infty} \sup_{y\in B_{\delta^{-1}\ln 2}} v_\mu(t,X(t)+y) \le \frac 12.
\]
Then since $v_\mu(t,X(t))=h^{-1}(\eps)$, applying Theorem \ref{ZlatosThm:tf}(i) shows that $u_\mu(t,X(t))\ge \eps$ and 
\[
\lim_{t\to\infty} \sup_{y\in B_{\delta^{-1}\ln 2}} u_\mu(t,X(t)+y) \le \frac 12.
\]
This shows that $L_\eps$ from \eqref{1.1} with $u=u_\mu$
must satisfy $L_\eps\ge \delta^{-1}\ln 2$.  Since $\delta>0$ was arbitrary, such $L_\eps<\infty$ cannot exist and we are done.

\section{Proof of Theorem \ref{ourmainthm}(iii)}\label{s:(iii)}

Throughout this section, $\interior(E)$  and $\partial E$ denote the interior and boundary of a set $E\subseteq \RR^{d}$.
We split the proof in two parts.

\subsection{Proof that \texorpdfstring{$u_\mu$}{} is not a transition front}
This follows immediately from Theorem \ref{ZlatosThm:tf}(ii) and the following result.

\begin{prop}\label{ReformulateNecessaryCondition}
	If $\mu$ satisfies \eqref{h:3}, then $0\in\ch(\mu)$.
	\end{prop}

The proof of Proposition \ref{ReformulateNecessaryCondition} uses several results from convex analysis:

\begin{lem}[Section 9, Chapter 6, Theorem 3 in \cite{conv}]\label{ImportantCor}
	Let $S\subseteq \mathbb{R}^d$ be a nonempty compact set. Then $0\not\in\ch(S)$ if and only if there exists an $\zeta\in S^{d-1}$ such that $S\subseteq \interior(\mathcal{W}_{0, \zeta})$. 
\end{lem}

\begin{lem}[Theorem $\Delta_{n}$ in \cite{gustin}] \label{VeryUsefulTheorem}
	If $S\subseteq \mathbb{R}^d$ and $x\in\interior(\ch(S))$, then there is  $S^*\subseteq S$ such that $\card(S^*)\le 2d$ and $x\in\interior(\ch(S^*))$.
\end{lem}

Finally, we need a technical result concerning stability of the convex hull of a finite set of points, which we prove in the Appendix.

\begin{prop}\label{ConvexHullWiggle}
If $S^* = \{x_1,\dots,x_k\}\subseteq\mathbb{R}^d$ and $0 \in \interior(\ch(S^*))$, then there is $\epsilon>0$ such that for all $y_i\in B_\epsilon(x_i)$, we have $0 \in \ch(\{y_1,\dots,y_k\})$.
\end{prop}

\begin{proof}[Proof of Proposition \ref{ReformulateNecessaryCondition}]
	By \eqref{h:3}, $\supp(\mu)\not\subseteq \interior(\mathcal{W}_{0, \zeta})$ for any $\zeta\in S^{d-1}$. Since $\supp(\mu)$ is compact, Lemma \ref{ImportantCor} implies that $0\in\ch(\supp(\mu))$.  We cannot have $0\in\partial(\ch(\supp(\mu)))$ because then convexity of $\ch(\supp(\mu))$ would imply  existence of a supporting hyperplane $H$ of $\ch(\supp(\mu))$ such that $0\in H$ (and then $H=\partial \mathcal{W}_{0, \zeta}$  for some $\zeta\in S^{d-1}$). This implies that $\supp(\mu)\subseteq \ch(\supp(\mu))\subseteq \mathcal{W}_{0, \zeta}$, yielding a contradiction.  Therefore $0\in\interior(\ch(\supp(\mu)))$, and Lemma \ref{VeryUsefulTheorem} shows that there exist $k\le 2d$ points $\{x_1,\dots,x_k\}\subseteq\ \ch(\supp(\mu))$
	such that 
	\[0\in\interior(\ch(\{x_1,\dots,x_k\})).\]

By Proposition \ref{ConvexHullWiggle}, there is $\ep>0$ such that $0\in \ch(\left\{y_{1}, y_{2}, \ldots, y_{k}\right\})$ whenever $y_{i}\in B_{\ep}(x_{i})$ for each $i=1,\dots,k$.  Since any $A\in \esssup(\mu)$ satisfies $A\cap B_{\epsilon}(x_i)\ne\emptyset$ for each $i=1,\dots,k$ (because $x_i\in \supp(\mu)$ and so $\mu(B_{\epsilon}(x_i))>0$), it follows that $0\in\ch(A)$.  Therefore,   $0\in\ch(\mu)$.
\end{proof}

\subsection{Proof that \texorpdfstring{$u_{\mu}$}{} is a transition solution with bounded width}

Let us start with some preliminary lemmas.  Note that we obviously have $\mu(\mathcal{W}^c_{0, \zeta})>0$ for any $\zeta\in S^{d-1}$.

\begin{lem}\label{InfMeasureHalfSpaces}
If $\mu$ satisfies \eqref{h:3}, then 
\begin{equation*}
a^*:=\inf\limits_{\zeta\in S^{d-1}}\mu(\mathcal{W}^c_{0, \zeta})>0. 
\end{equation*}
\end{lem}
\begin{proof}
If $a^*=0$, then there is a sequence $\{\zeta_n\}\subseteq S^{d-1}$ with  $\mu(\mathcal{W}^c_{0, \zeta_n})<2^{-n}$ for each $n$.  By compactness of $S^{d-1}$, after passing to a subsequence we can assume that $\zeta_{n}\rightarrow \zeta\in S^{d-1}$. But 
\begin{equation*}
\mathcal{W}^c_{0, \zeta}\subseteq \bigcap\limits_{j=1}^{\infty}\bigcup\limits_{n=j}^{\infty}\mathcal{W}^c_{0, \zeta_{n}}
\end{equation*}
then yields $\mu(\mathcal{W}^c_{0, \zeta})=0$, a contradiction with \eqref{h:3}.
\end{proof}

For $N\ge 1$, let
\begin{align*}
Z_{N}:=\{\zeta\in S^{d-1}:\mu(C_{N, \zeta})>0\},
\end{align*}
where for $\zeta\in S^{d-1}$ we let  
\begin{align*}
C_{N, \zeta}:= \interior(\mathcal{W}_{N^{-1},-\zeta}\cap A(N^{-1},1)).
\end{align*}

\begin{lem}\label{UniformCutoutPosMeasure}
If $\mu$ satisfies \eqref{h:3}, then $S^{d-1}=Z_{N}$ for some $N\ge 1$.
\end{lem}

\begin{proof}
Note that $Z_{N}$ is open in $S^{d-1}$ for each $N\ge 1$ because we have $C_{N, \zeta}\subseteq \bigcup_{n=1}^\infty C_{N, \zeta_n}$ whenever $\zeta_n\to\zeta$.  
Since obviously $Z_{N}\subseteq Z_{N+1}$ for each $N$,
it follows that $\{Z_{N}^c\}_{N=1}^\infty$ is a decreasing sequence of compact sets.  If none of these is empty, then there exists $\zeta\in S^{d-1}\setminus \bigcup_{N=1}^\infty Z_N$, which contradicts $\mathcal{W}^c_{0, \zeta}=\bigcup_{N=1}^\infty C_{N,\zeta}$ and $\mu(\mathcal{W}^c_{0, \zeta})>0$.
\end{proof}

From this, similarly to Lemma \ref{InfMeasureHalfSpaces}, we obtain the following.

\begin{lem}\label{InfMeasureCutouts}
If $\mu$ satisfies \eqref{h:3} and $N$ is from Lemma \ref{UniformCutoutPosMeasure}, then 
\[
b^*:=\inf_{\zeta\in S^{d-1}}\mu(C_{N,\zeta})>0.
\] 
\end{lem}

From now on, we fix $N$ from Lemma \ref{UniformCutoutPosMeasure} and $b^*$ from Lemma \ref{InfMeasureCutouts} (both depending on $\mu$).  We will now prove \eqref{1.1} for $u_\mu$,
first considering all large negative $t$.  

\begin{lem}\label{LevelSetEstimate}
If $\mu$ satisfies \eqref{h:3} and $\ep\in(0,\frac{1}{2})$, then there are $K,T>0$ such that $u_{\mu}(t,x)\ge 1-\ep$ whenever $t\le -T$ and $\vert x\vert\ge K\vert t\vert$.
\end{lem}

\begin{proof}
Let $K:=3N^{2}$ and $T:=\ln \frac{h^{-1}(1-\ep)}{b^*}.$
Since for any $x\in\mathbb R^d\setminus\{0\}$ we obviously have
	\begin{align*}
	\inf_{\xi\in C_{N,x\vert x\vert^{-1}}}\left(-\xi\cdot \frac{x}{\vert x\vert}\right)\ge \frac{1}{N^{2}},
	\end{align*}
 for any $t\le -T$ and $x\in\mathbb{R}^d$ with $\vert x\vert \ge K\vert t\vert$ we obtain
	\begin{align*}
	v_{\mu}(t,x)
	&=\int_{B_1}e^{-\xi\cdot x+(|\xi|^{2}+1)t}\,d\mu(\xi)\\
	&\ge \int_{C_{N,x\vert x\vert^{-1}}} e^{ \vert x\vert N^{-2}+2t}\,d\mu(\xi)\\
	&\ge e^{(KN^{-2} - 2)\vert t\vert} \mu(C_{N,x\vert x\vert^{-1}}) \\
	&\ge e^{T} b^*\\
	&= h^{-1}(1-\ep).
	\end{align*}
	Theorem \ref{ZlatosThm:tf}(i) now finishes the proof.
\end{proof}

\begin{lem}\label{IntegralSmallNear0}
If $\mu$ satisfies \eqref{h:3} and $\ep\in(0,\frac{1}{2})$, then the following holds for any $a>0$.  There are $T_a,\delta_a>0$ such that if $(t,x)\in (-\infty, -T_{a}]\times \RR^{d}$ 
and $u_{\mu}(t,x)<1-\ep$, then 
	\begin{align*}
	\int_{B_{\delta_{a}}} e^{-\xi\cdot x+(|\xi|^2+1)t}\,d\mu(\xi)\le a.
	\end{align*}
\end{lem}

\begin{proof}
We can assume without loss that $a\le\mu(B_1)$.  Let $K,T$ be from Lemma \ref{LevelSetEstimate} and define
\begin{align*}
T_a &:= \max\left\{ T, 1+\left| \ln \frac a{\mu(B_1)} \right| \right\},
\\ \delta_a &:= \frac 1K \left(1-\frac{\left|\ln\frac{a}{\mu(B_1)}\right|}{T_a} \right)>0.
\end{align*}
Since $T_a\ge T$, Lemma \ref{LevelSetEstimate} and Theorem \ref{ZlatosThm:tf}(i) show that for any $(t,x)$ as above, we must have $\vert x\vert < K\vert t\vert$.

We also have $\delta_a K-1<0$, hence for any $t\le -T_a$ we find 
		\begin{align*}
		(\delta_a K-1)\vert t\vert \le (\delta_a K-1)T_a\le \ln\frac{a}{\mu(B_1)}.
		\end{align*}
		It follows that for $(t,x)$ as above we obtain 
\[
\int_{B_{\delta_a}} e^{-\xi\cdot x+(|\xi|^2+1)t}\,d\mu(\xi)
		\le \int_{B_{\delta_a}} e^{\delta_a K\vert t\vert +t}\,d\mu(\xi)
		\le e^{(\delta_a K-1)\vert t\vert} \mu(B_1)
		\le a,
\]
		and the proof is finished.
	\end{proof}

We can now prove \eqref{1.1} for $u=u_\mu$.

\begin{prop}\label{NegTime}
	If $\mu$ satisfies \eqref{h:3} and $\ep\in(0,\frac{1}{2})$, then there is $L_{\ep}<\infty$ such that  for each $t\in\mathbb R$,
	\begin{align*}
	\Omega_{u_{\mu},\ep}(t)\subseteq B_{L_{\ep}}(\Omega_{u_{\mu},1-\ep}(t)).
	\end{align*}
\end{prop}

\begin{proof}
For any $\zeta\in S^{d-1}$, let
\begin{align*}
Y_{\zeta}:=\left\{\xi\in B_1: \zeta\cdot\xi \ge \frac {|\xi|}2 \right\}=\mathcal W_{1/2,\zeta}\cap B_1.
\end{align*}
Let also $a:= \frac\ep 2|Y_\zeta||B_1|^{-1}$ (note that $|Y_\zeta|$ is independent of $\zeta$) and let $\delta_a,T_{a}$ be from Lemma \ref{IntegralSmallNear0}.

We will first consider times $t\le -T_{a}$. 
Fix any such $t$ and let $x$ be such that $ u_{\mu}(t,x)\ge \ep$. 
Since $v_\mu(t,x)\ge u_\mu(t,x)\ge \ep$, there must be $\zeta\in S^{d-1}$ such that
	\begin{align*}
	\int_{Y_{\zeta}} e^{-\xi\cdot x + (|\xi|^2+1)t}\,d\mu(\xi)\ge 2a.
	\end{align*} 
Then Lemma \ref{IntegralSmallNear0} shows that
	\begin{align*}
	\int_{Y_{\zeta}\cap A(\delta_a,1)} e^{-\xi\cdot x + (|\xi|^2+1)t}\,d\mu(\xi)&\ge \int_{Y_{\zeta}} e^{-\xi\cdot x + (|\xi|^2+1)t}\,d\mu(\xi)-\int_{B_{\delta_a}} e^{-\xi\cdot x + (|\xi|^2+1)t}\,d\mu(\xi)
\\
&\geq 2a-a=a,
	\end{align*}
hence for $L_{\ep}^{-}:=\frac 2{\delta_a} \ln \frac{h^{-1}(1-\ep)}{a}$ we have
	\begin{align*}
	v_{\mu}(t,x-L_{\ep}^{-}\zeta)
	&=\int_{B_1}e^{-\xi\cdot(x-L_{\ep}^{-}\zeta)+(|\xi|^2+1)t}\, d\mu(\xi)\\
	&\ge \int_{Y_{\zeta}\cap A(\delta_{a},1)} e^{L_{\ep}^{-}(\xi\cdot\zeta)} e^{-\xi\cdot x+(|\xi|^2+1)t}\, d\mu(\xi)\\
&\ge e^{L_{\ep}^{-} \delta_a/2} a = h^{-1}(1-\ep).
	\end{align*}
If we now choose $L_{\ep}\ge L_{\ep}^{-}$, from  Theorem \ref{ZlatosThm:tf}(i) we obtain the claim for all $t\le -T_a$.

Let us now consider $t>-T_a$. For each $\zeta\in S^{d-1}$ 
we obviously have
\[\inf_{\xi\in C_{N,\zeta}} (- \xi\cdot \zeta) \geq \frac{1}{N^{2}}.\]  
Then for each $s\ge L_{\ep}^+:= N^{2}\left(\left|\ln\frac{h^{-1}(1-\ep)}{b^*}\right|+2T_a\right)$ and $t> -T_a$ we have
\[
v_{\mu}(t,s\zeta)
\ge\int_{C_{N,\zeta}}e^{-\xi\cdot s\zeta-2T_a}\,d\mu(\xi)
\ge e^{sN^{-2}-2T_a}\mu(C_{N,\zeta})\ge e^{sN^{-2}-2T_a} b^*\ge h^{-1}(1-\ep).
\]
Theorem \ref{ZlatosThm:tf}(i) then yields $u_\mu(t,s\zeta)\ge 1-\ep$ for all $\zeta\in S^{d-1}$, $s\ge L_{\ep}^+$, and $t>-T_a$. Hence
\[
B_{L_{\ep}^+}^{c} \subseteq \Omega_{u_{\mu},1-\ep}(t)
\]
for all $t>-T_a$, and  the result follows with $L_\ep:=\max\{L_\ep^-,L_\ep^+\}$.
\end{proof}

Since each $u_\mu$ is a transition solution
it follows that $u_\mu$ is indeed a transition solution with bounded width.

\section{Appendix} 

In this appendix, we prove Proposition \ref{ConvexHullWiggle}.
The proof uses two auxiliary lemmas:

\begin{lem}\label{WiggleLemma1}
	If $0 \in \interior(\ch(\{x_1,\dots,x_k\}))$, then there are $c_i>0$ such that \[\sum_{i=1}^k c_ix_i =0.\]
\end{lem}

\begin{proof}
	There obviously are $a_i \geq 0$ with $\sum_{i=1}^k a_i = 1$ such that $\sum_{i=1}^k a_i x_i=0$. Since $0 \in \interior(\ch(\{x_1,\dots,x_k\}))$, there is $\delta>0$ such that we have $-\delta x_{i}\in \ch(\{x_1, \dots, x_k\})$ for each $i$. Thus, each $-\delta x_i$ may be written as a convex combination $-\delta x_i = \sum_{j=1}^k b_{ij} x_j$, with $b_{ij}\geq 0$ and $\sum_{j=1}^k b_{ij} = 1$. Then
	\begin{align*}
	0 & = \sum_{i=1}^k a_i x_i  +\sum_{i=1}^k \delta x_i + \sum_{i=1}^k -\delta x_i\\
	&  = \sum_{i=1}^k (a_i + \delta) x_i + \sum_{i=1}^k \sum_{j=1}^k b_{ij} x_j\\
	& = \sum_{i=1}^k (a_i +\delta + \sum_{j=1}^k b_{ji}) x_i.
	\end{align*}
	Hence we can take $c_i := a_i +\delta + \sum_{j=1}^k b_{ji}>0$.
\end{proof}

\begin{lem}\label{WiggleLemma2}
	If $0\in \interior (\ch(\{x_1, \dots, x_k\}))$, then for any $r>0$ there is $\epsilon>0$ such that any $p\in B_\epsilon$ can be written as
	\[p = \sum_{i=1}^k a_i x_i, \qquad \text{where }|a_i|\le r.\]
\end{lem}

\begin{proof}
	There is $\delta>0$ such that $B_\delta\subseteq \ch(\{x_1,\dots, x_k\})$. Then each $z\in B_\delta$ can be written as $z = \sum_{i=1}^k b_i x_i$,
	where $b_i\geq 0$ and $\sum_{i=1}^k b_i = 1$. Given $r>0$, let $\epsilon := r\delta$. Then for any $p\in B_\epsilon$ we have $p = rz$ for some $z\in B_\delta$, so $p = \sum_{i=1}^k (rb_i) x_i$
	with some $b_i\in[0,1]$.  The proof is finished.
\end{proof}

\begin{proof}[Proof of Proposition \ref{ConvexHullWiggle}]
Let $c_i>0$ ($i=1,\dots,k$) be as in Lemma \ref{WiggleLemma1}.  Consider the system of linear equations
	\[A\Theta:=
	\begin{bmatrix}
	1+a_{11} & a_{12} & \dots & a_{1k}\\
	a_{21} & 1+a_{22} & \dots & a_{2k}\\
	\vdots & \vdots & \ddots & \vdots\\
	a_{k1} & a_{k2} & \dots & 1+a_{kk}
	\end{bmatrix} \begin{bmatrix}
	\theta_1\\
	\theta_2\\
	\vdots\\
	\theta_k
	\end{bmatrix} = \begin{bmatrix}
	c_1\\
	c_2\\
	\vdots\\
	c_k
	\end{bmatrix}\]
	with some given $a_{ij}\in \mathbb{R}$. The determinant \[\det A =\begin{vmatrix}
	1+a_{11} & a_{12} & \dots & a_{1k}\\
	a_{21} & 1+a_{22} & \dots & a_{2k}\\
	\vdots & \vdots & \ddots & \vdots\\
	a_{k1} & a_{k2} & \dots & 1+a_{kk}
	\end{vmatrix}\] is a continuous  function of the $a_{ij}$ and equals 1 when they all vanish.  Thus,  there is $r_0>0$ such that $\max_{i,j}|a_{ij}|\le r_0$ implies $\det A>0$. Similarly, 
	\[\det M_l = \begin{vmatrix}\
	1+a_{11} & \dots & a_{1(l-1)} & c_1 & a_{1(l+1)} & \dots & a_{1k}\\
	a_{21} & \dots & a_{2(l-1)} & c_2 & a_{2(l+1)}&\dots & a_{2k}\\
	\vdots & \ & \vdots & \vdots  &\vdots & \ & \vdots\\
	a_{k1} & \dots & a_{k(l-1)}& c_k & a_{k(l+1)} & \dots & 1+a_{kk}
	\end{vmatrix}\]
depends continuously on the $a_{ij}$ and equals $c_l>0$ when they all vanish.
Thus,	 there is $r_1>0$ such that $\max_{i,j}|a_{ij}|\le r_1$ implies  $\max_l \det M_l >0$.

	Let $r := \min \{r_0, r_1\}$, and let $\epsilon>0$ be as in Lemma \ref{WiggleLemma2}.  Let $y_j \in B_\epsilon(x_j)$ be arbitrary and denote $p_j: = y_j -x_j$. Then Lemma \ref{WiggleLemma2} shows that each $p_j$ can be written as
	\[p_j = \sum_{i=1}^k a_{ij} x_i, \qquad \text{with } |a_{ij}|\le r.\]
Finally, for each $j=1,\dots,k$ let 
\[\theta_j := \frac{\det M_j}{\det A}>0,\]
so that $\Theta=(\theta_1, \dots, \theta_k)$ is the (unique) solution of the above system (by Cramer's rule).
Then
	\begin{align*}
	\sum_{j=1}^k \theta_j y_j & = \sum_{j=1}^k \theta_j (x_j + p_j)\\
	& = \sum_{j=1}^k \theta_j x_j + \sum_{j=1}^k \theta_j \sum_{i=1}^k a_{ij} x_i\\
	& = \sum_{i=1}^k \theta_i x_i + \sum_{i=1}^k \Big[\sum_{j=1}^k a_{ij}\theta_j x_i\Big]\\
	& = \sum_{i=1}^k \Big[\theta_i + \sum_{j=1}^k a_{ij}\theta_j\Big] x_i\\
	& = \sum_{i=1}^k c_i x_i = 0.
	\end{align*}
	 Normalizing now yields the desired convex combination
	\[0 =\sum_{j=1}^k \frac{\theta_j}{\sum_{i=1}^k \theta_i} y_j,\]
	and the proof is finished.
\end{proof}

\end{document}